\documentclass[]{article}
\usepackage{amssymb,amsmath,amsthm,mathdesign}
\usepackage{pdfsync}

\newcommand{\R}{\mathbf{R}}

\newcommand {\E}{\mathrm{E}}

\renewcommand{\d}{\text{\rm d}}

\newcommand{\sL}{\mathcal{L}}

\newcommand{\sG}{\mathcal{G}}

\newcommand{\sA}{\mathcal{A}}
\newcommand{\sE}{\mathcal{E}}
\newcommand{\sI}{\mathcal{I}}
\newcommand{\sS}{\mathcal{S}}

\newcommand{\ve}{\epsilon}

\newtheorem{stat}{Statement}[section]
\newtheorem{proposition}[stat]{Proposition}
\newtheorem{corollary}[stat]{Corollary}
\newtheorem{theorem}[stat]{Theorem}
\newtheorem{lemma}[stat]{Lemma}

\theoremstyle{definition}
\newtheorem{definition}[stat]{Definition}\newtheorem{remark}[stat]{Remark}
\newtheorem{example}[stat]{Example}

\numberwithin{equation}{section}

\begin{document}

\title{\bf On some properties of a class of fractional stochastic heat equations.%
	\thanks{%
	Research supported in part by EPSRC.}}
	
\author{Mohammud Foondun\\Loughborough University
\and Kuanhou Tian\\ Loughborough University \and Wei Liu \\Loughborough University}

\date{April 27, 2014}
\maketitle
\begin{abstract}
	We consider nonlinear parabolic
	stochastic equations of the form
	$\partial_t u=\sL u + \lambda \sigma(u)\dot \xi$ on the ball $B(0,\,R)$, where
	$\dot \xi$ denotes some Gaussian noise and
	$\sigma$ is Lipschitz continuous. Here $\sL$ corresponds to an $\alpha$-stable process killed upon exiting $B(0, R)$. We will consider two types of noise; space-time white noise and spatially correlated noise.  Under a linear growth condition on $\sigma$, we study growth properties of the second moment of the solutions. Our results are significant extensions of  those in \cite{foonJose} and complement those of \cite{Khoshnevisan:2013ab} and \cite{Khoshnevisan:2013aa}.

\vskip .2cm \noindent{\it Keywords:}
		Stochastic partial differential equations, \\
		
	\noindent{\it \noindent AMS 2000 subject classification:}
		Primary: 60H15; Secondary: 82B44.\\
		
	\noindent{\it Running Title:} Noise excitability and
		parabolic SPDEs.\newpage
\end{abstract}

\section{Introduction and main results.}

Consider the following stochastic heat equation on the interval $(0,1)$ with Dirichlet boundary condition:
\begin{equation*}
\left|\begin{split}
&\partial_t u_t(x)=\frac{1}{2}\partial _{xx}u_t(x)+\lambda u_t(x)\dot{w}(t,\,x)\quad\text{for}\quad0<x<1\quad\text{and}\quad t>0\\
&u_t(0)=u_t(1)=0 \quad \text{for}\quad t>0.
\end{split}
\right.
\end{equation*}
Here $\dot{w}$ denotes white noise, $\lambda$ is a positive parameter and $u_0(x)$ is the initial condition. Set 
\begin{equation*}
\sE_t(\lambda):=\sqrt{\int_{0}^1\E|u_t(x)|^2 \d x}.
\end{equation*}
The study of $\sE_t(\lambda)$ as $\lambda$ gets large was initiated in  \cite{Khoshnevisan:2013aa} and \cite{Khoshnevisan:2013ab}. In \cite{foonJose}, it was shown that $\sE_t(\lambda)$ grows like $\text{const}\times\exp(\lambda^4)$ as $\lambda$ gets large.  The main aim of this paper is to extend similar results to a much wider class of stochastic equations. Existence and uniqueness of solutions to these equations are known, so we refer the reader to \cite{walsh} and \cite{minicourse} for technical details about this issue.
We will first look at equations driven by white noise. Fix $R>0$ and consider the following:

\begin{equation}\label{eq:dirichlet}
\left|\begin{split}
&\partial_t u_t(x)=\sL u_t(x)+\lambda \sigma(u_t(x))\dot{w}(t,\,x),\\
&u_t(x)=0, \quad \text{for all}\quad x\in B(0,\,R)^c,
\end{split}
\right.
\end{equation}
where $\dot{w}$ denotes white noise $(0,\infty)\times B(0,\,R)$.  Here and throughout this paper, the initial function $u_0:B(0,\,R)\rightarrow \R_+$ is a nonrandom nonnegative function which is strictly positive in the set of positive measure in $B(0\,,R)$. $\sigma:\R^d \rightarrow \R$ is a continuous function with $\sigma(0)=0$ and
\begin{equation*}
l_\sigma:=\inf_{x\in \R^d\backslash \{0\}}\left| \frac{\sigma(x)}{x}\right|\quad \text{and}\quad
L_\sigma:=\sup_{x\in \R^d\backslash \{0\}}\left| \frac{\sigma(x)}{x}\right|,
\end{equation*}
where $0<l_\sigma\leq L_\sigma<\infty$.  $\sL$ is the generator of an $\alpha$-stable process killed upon exiting $B(0,\,R)$ so that \eqref{eq:dirichlet} can be thought of as the Dirichlet problem for fractional Laplacian of order $\alpha$.

Following Walsh \cite{walsh}, we say that $u$ is a {\it mild solution} to \eqref{eq:dirichlet} if it satisfies the following evolution equation,
\begin{equation}\label{mild:dirichlet}
u_t(x)=
(\sG_Du)_t(x)+ \lambda \int_{B(0,\,R)}\int_0^t p_D(t-s,\,x,\,y)\sigma(u_s(y))w(\d s\,\d y),
\end{equation}
where
\begin{equation*}
(\sG_D u)_t(x):=\int_{B(0,\,R)} u_0(y)p_D(t,\,x,\,y)\,\d y.
\end{equation*}
Here $p_D(t,\,x,\,y)$ denotes the fractional Dirichlet heat kernel. It is also well known that this unique mild solution satisfies the following integrality condition 
\begin{equation}\label{moments}
\sup_{x\in B(0,\,R)}\sup_{t\in[0,\,T]} \E|u_t(x)|^k<\infty \quad\text{for all}\quad T>0 \quad\text{and}\quad k\in[2,\,\infty],
\end{equation}
which imposes the restriction that $d=1$, which will be in force whenever we deal with \eqref{eq:dirichlet}.

Here is our first result.

\begin{theorem}\label{white}
Fix $\ve>0$ and let $x\in B(0,\,R-\ve)$, then for any $t>0$,
\begin{equation*}
\lim_{\lambda\rightarrow \infty}\frac{\log \log \E|u_t(x)|^2}{\log \lambda}=\frac{2\alpha}{\alpha-1},
\end{equation*}
where $u_t$ is the unique solution to \eqref{eq:dirichlet}.
\end{theorem}
Set
\begin{equation}\label{energy}
\sE_t(\lambda):=\sqrt{\int_{B(0,\,R)}\E|u_t(x)|^2 \d x}.
\end{equation}
We have the following definition. 
\begin{definition}
The {\it excitation index} of $u$ at time $t$ is given by 
\begin{equation*}
e(t):=\lim_{\lambda\rightarrow \infty}\frac{\log \log \sE_t(\lambda)}{\log \lambda}
\end{equation*}
\end{definition}
We then have the following corollary.
\begin{corollary}\label{cor:white}
The excitation index of the solution to \eqref{eq:dirichlet} is $\frac{2\alpha}{\alpha-1}$.
\end{corollary}
It can be seen that when $\alpha = 2$ this gives the result in \cite{foonJose}.  Our second main result concerns coloured noise driven equations.  Consider
\begin{equation}\label{eq:dirichlet:colored}
\left|\begin{split}
&\partial_t u_t(x)=\sL u_t(x)+\lambda \sigma(u_t(x))\dot{F}(t,\,x),\\
&u_t(x)=0, \quad \text{for all}\quad x\in B(0,\,R)^c.
\end{split}
\right.
\end{equation}
This equation is exactly the same as \eqref{eq:dirichlet} except for the noise which is now given by $\dot{F}$ and can be described as follows.

\begin{equation*}
\E[\dot{F}(t,\,x)\dot{F}(s,\,y)]=\delta_{0}(t-s)f(x,y),
\end{equation*}
where $f$ is given by the socalled Riesz kernel:
\begin{equation*}
f(x,\,y):=\frac{1}{|x-y|^\beta}.
\end{equation*}
Here $\beta$ is some positive parameter satisfying $\beta<d$.  Other than the noise term, we will work under the exact conditions as those for equation \eqref{eq:dirichlet}. The mild solution will thus satisfy the following integral equation.
\begin{equation}\label{mild:dirichlet:colored}
u_t(x)=
(\sG_Du)_t(x)+ \lambda \int_{B(0,\,R)}\int_0^t p_D(t-s,\,x,\,y)\sigma(u_s(y))F(\d s\,\d y).
\end{equation}

Existence-uniqueness considerations will force us to further impose $\beta<\alpha$, see for instance \cite{ferrante}.  Our first result concerning \eqref{eq:dirichlet:colored} is the following.

\begin{theorem}\label{coloured}
Fix $\ve>0$ and let $x\in B(0,\,R-\ve)$, then for any fixed $t>0$,
\begin{equation*}
\lim_{\lambda\rightarrow \infty}\frac{\log \log \E|u_t(x)|^2}{\log \lambda}=\frac{2\alpha}{\alpha-\beta},
\end{equation*}
where $u_t$ is the unique solution to \eqref{eq:dirichlet:colored}.
\end{theorem}

\begin{corollary}\label{cor:coloured}
The excitation index of the solution to \eqref{eq:dirichlet:colored} is $\frac{2\alpha}{\alpha-\beta}$.
\end{corollary}

It is clear that our results are significant extensions of those in \cite{foonJose} and \cite{Khoshnevisan:2013ab}. The techniques are also considerably harder and required some new highly non-trivial ideas which we now mention.
\begin{itemize}
\item We need to compare the heat kernel estimates for killed stable process with that of ``unkilled" one. To do that, we will need sharp estimates of the Dirichlet heat kernel.
\item We will also need to study some renewal-type inequalities and by doing so, we come across the Mittag-Leffler function whose asymptotic properties become crucial.
\item While the above two ideas are enough for the proof of Theorem \ref{white}, we will also need to significantly modify the localisation techniques of \cite{Khoshnevisan:2013ab} to complete the proof of Theorem \ref{coloured}.
\end{itemize}

Our method seems suited for the study of a much wider class of equations. To illustrate this, we devote a section to various extensions.

Here is a plan of the article. In section 2, we collect some information about the heat kernel and the renewal-type inequalities.  In section 3, we prove the main results concerning \eqref{eq:dirichlet}. Section 4 contains analogous proofs for \eqref{eq:dirichlet:colored}.  In section 5, we extend our study to a much wider class of equations.

Finally, throughout this paper, the letter $c$ with or without subscripts will denote constants whose exact values are not important to us and can vary from line to line.

\section{Preliminaries.}

Let $X_t$ denote the $\alpha$-stable process on $\R^d$ with $p(t,\,x,\,y)$ being its transition density. It is well known that
\begin{equation*}
c_1\left(t^{-d/\alpha}\wedge \frac{t}{|x-y|^{d+\alpha}}\right)\leq p(t,\,x,\,y)\leq c_2\left(t^{-d/\alpha}\wedge \frac{t}{|x-y|^{d+\alpha}}\right),
\end{equation*}
where $c_1$ and $c_2$ are positive constants.  We define the first exit of time $X_t$ from the ball $B(0,\,R)$ by
\begin{equation*}
\tau_{B(0,\,R)}:=\inf\{t>0, X_t\notin B(0,\,R) \}.
\end{equation*}
We then have the following representation for $p_D(t,\,x,\,y)$
\begin{equation*}
p_D(t,\,x,\,y)=p(t,\,x,\,y)-\E^x[p(t-\tau_{B(0,\,R)}, X_{\tau_{B(0,\,R)}}, y);\tau_{B(0,\,R)}<t].
\end{equation*}
From the above, it is immediate that
\begin{equation*}
p_D(t,\,x,\,y)\leq p(t,\,x,\,y) \quad\text{for all}\quad x,\,y\in \R^d.
\end{equation*}
This in turn implies that 
\begin{equation}\label{heat:upper}
p_D(t,\,x,\,y)\leq \frac{c_1}{t^{d/\alpha}}\quad\text{for all}\quad x,\,y\in \R^d.
\end{equation}
We now provide some sort of converse to the above inequality.  Not surprisingly, this inequality will hold for small times only.

\begin{proposition}\label{lower}
Fix $\ve>0$. Then for all $x,\,y\in B(0,\,R-\ve)$, there exists a $t_0>0$ and a constant $c_1$, such that
\begin{equation*}
p_D(t,\,x,\,y)\geq c_1 p(t,\,x,\,y),
\end{equation*}
whenever $t\leq t_0$. And if we further impose that  $|x-y|\leq t^{1/\alpha}$, we obtain the following
\begin{equation}\label{heat:lower}
p_D(t,\,x,\,y)\geq\frac{c_2}{t^{d/\alpha}},
\end{equation}
where $c_2$ is some positive constant.
\end{proposition}

\begin{proof}
Set $\delta_{B(0,\,R)}(x):=\text{dist}(x, B(0,\,R)^c)$.  It is known that
\begin{equation*}
p_{D}(t,\,x,\,y)\geq c_1\left(1\wedge \frac{\delta^{\alpha/2}_{B(0,\,R)}(x)}{t^{1/2}}\right)\left(1\wedge \frac{\delta^{\alpha/2}_{B(0,\,R)}(y)}{t^{1/2}}\right)p(t,\,x,\,y),
\end{equation*}
for some constant $c_1$. See for instance \cite{Bogdan} and references therein. Since $x\in B(0,\,R-\ve)$, we have $\delta_{B(0,\,R)}(x)\geq \ve$. Now choosing $t_0=\ve^\alpha$, we have $\delta^{\alpha/2}_{B(0,\,R)}(x)\geq t^{1/2}$ for all $t\leq t_0$. Similarly, we have $\delta^{\alpha/2}_{B(0,\,R)}(y)\geq t^{1/2}$ which together with the above display yield
\begin{equation*}
p_{D}(t,\,x,\,y)\geq c_2p(t,\,x,\,y)\quad \text{for all} \quad x,y\in B(0,\,R-\ve),
\end{equation*}
whenever $t\leq t_0$. We now use the fact that

\begin{equation*}
p(t,\,x,\,y)\geq c_3\left(\frac{t}{|x-y|^{d+\alpha}}\wedge t^{-d/\alpha} \right).
\end{equation*}
to end up with \eqref{heat:lower}.
\end{proof}
We now make a simple remark which will be important in the sequel.

\begin{remark}\label{deterministic}
Recall that for any $t>0$ and $x\in B(0,\,R)$.
\begin{equation*}
(\sG_Du)_t(x):=\int_{B(0,\,R)}p_D(t,\,x,\,y)u_0(y)\d y.
\end{equation*}
Fix $\ve>0$ and set $g_t:=\inf_{x\in B(0,\,R-\ve)} \inf_{s\leq t}(\sG_Du)_s(x)$. Then for any fixed $t>0$, we have $g_t=\inf_{x\in B(0,\,R-\ve)}(\sG_Du)_t(x)$  and $g_t>0$. This is because of $(\sG_Du)_s(x)$ is a decreasing function of $s$ and $p_D(t,\,x,\,y)$ is strictly positive for all $x,\,y\in B(0,\,R-\ve)$.
\end{remark}

We now give a definition of the Mittag-Leffler function which is denoted by $E_\beta$ where $\beta$ is some positive parameter.  Define

\begin{equation*}
E_\beta(t):=\sum_{n=0}^\infty \frac{t^{n}}{\Gamma(n\beta+1)}\quad\text{for}\quad t>0.
\end{equation*}

This function is well studied and crops up in a variety of settings including the study of fractional equations \cite{Mainardi}. In our context, we encounter it in the study of the renewal inequalities mentioned in the introduction.  Even though, a lot is known about this function, we will need the following simple fact whose statement is motivated by the use we make of it later.

\begin{proposition}\label{boundMLf}
For any fixed $t>0$, we have

\begin{equation*}
\limsup_{\theta\rightarrow \infty}\frac{\log\log E_\beta(\theta t)}{\log \theta}\leq \frac{1}{\beta},
\end{equation*}
and
\begin{equation*}
\liminf_{\theta\rightarrow \infty}\frac{\log\log E_\beta(\theta t)}{\log \theta}\geq \frac{1}{\beta}.
\end{equation*}
\end{proposition}

\begin{proof}
By using Laplace transforms techniques, one can show that
\begin{equation*}
| E_\beta(z) - \frac{1}{\beta} e^{z^{1/\beta}} | = o (e^{z^{1/\beta}}).
\end{equation*}
See for instance \cite{Henry} and references therein for more details. Thus for any positive constant $\epsilon > 0$ there exists a $Z > 0$ such that for all $z > Z$
\begin{equation*}
  | E_\beta(z) - \frac{1}{\beta} e^{z^{1/\beta}} | \leq \epsilon e^{z^{1/\beta}}.
\end{equation*}
Choosing $\epsilon < 1/\beta$, it it easy to see that
\begin{equation*}
\log(\log (\frac{1}{\beta} - \epsilon) + z^{1/\beta}) \leq \log \log E_\beta(z) \leq \log( \log(\frac{1}{\beta} + \epsilon) + z^{1/\beta}).
\end{equation*}
Letting $z= \theta t$, the above yield the assertions of the proposition.
\end{proof}

What follows is a consequence of Lemma 14.1 of \cite{Khoshnevisan:2013aa}. But for the sake of completeness, we give a quick proof based on the asymptotic behaviour of the Mittag-Leffler function which we used in the above proof.  Fix $\rho>0$ and consider the following:

\begin{equation}\label{sum}
S(t):=\sum_{k=1}^\infty \left( \frac{t}{k^\rho}\right)^{k}\quad\text{for}\quad t>0.
\end{equation}

\begin{lemma}\label{sum}
For any fixed $t>0$, we have
\begin{equation*}
\liminf_{\theta\rightarrow \infty}\frac{\log \log S(\theta t)}{\log \theta}\geq \frac{1}{\rho}.
\end{equation*}
\end{lemma}

\begin{proof}
From the asymptotic property of the Gamma function, there exists an $N>0$ such that for $k\geq N$, we have $\Gamma(k\rho+1)\geq \left(\frac{\rho k}{e}\right)^{\rho k}$. We thus have
\begin{equation*}
\begin{aligned}
S(t)&\geq\sum_{k=N}^\infty\left[\left(\frac{\rho}{e}\right)^\rho t\right]^{k}\frac{1}{\Gamma(k\rho+1)}\\
&=E_\rho[\left(\frac{\rho}{e}\right)^\rho t]-\sum_{k<N}\left[\left(\frac{\rho}{e}\right)^\rho t\right]^{k}\frac{1}{\Gamma(k\rho+1)}.
\end{aligned}
\end{equation*}
An application of Proposition \ref{boundMLf} proves the result.
\end{proof}

We now present the renewal inequalities.
\begin{proposition}\label{prop:upper-renew}
Let $T\leq \infty$ and $\beta>0$.  Suppose that $f(t)$ is a locally integrable function satisfying

\begin{equation}\label{renew:upper}
f(t)\leq c_1+\kappa\int_0^t(t-s)^{\beta-1}f(s)\d s\quad\text{for all}\quad0\leq t\leq T,
\end{equation}
where $c_1$ is some positive number.  Then for any $t\in(0,T]$, we have the following

\begin{equation*}
\limsup_{\kappa\rightarrow \infty}\frac{\log \log f(t)}{\log \kappa}\leq \frac{1}{\beta}.
\end{equation*}
\end{proposition}

\begin{proof}
We begin by setting $(\sA\psi)(t):=\kappa\int_0^t(t-s)^{\beta-1}\psi(s)\d s$ where $\psi$ can be any locally integrable function. And for any fixed integer $k>1$, we have $(\sA^k\psi)(t):=\kappa\int_0^t(t-s)^{\beta-1}(\sA^{k-1}\psi)(s)\d s$. We further set $1(s):=1$ for all $0\leq s\leq T$. With these notations, \eqref{renew:upper} can be succinctly written as $f(t)\leq c_1+(\sA f)(t)$
which upon iterating becomes
\begin{equation}\label{iteration}
f(t)\leq c_1\sum_{k=0}^{n-1}(\sA^k1)(t)+(\sA^n f)(t).
\end{equation}
Some further computations show that
\begin{equation*}
(\sA^n f)(t)=\frac{(\kappa\Gamma(\beta))^n}{\Gamma(n\beta)}\int_0^t(t-s)^{n\beta-1}f(s)\,\d s
\end{equation*}
and therefore we also have
\begin{equation*}
(\sA^n 1)(t)=\frac{(\kappa\Gamma(\beta))^nt^{n\beta}}{\Gamma(n\beta+1)}.
\end{equation*}
As $n\rightarrow \infty$, we have $(\sA^n f)(t)\rightarrow 0$. We thus end up with
\begin{equation*}
\begin{aligned}
f(t)&\leq c_1\sum_{k=0}^\infty(\sA^k1)(t)\\
&=c_1\sum_{k=0}^\infty \frac{(\kappa\Gamma(\beta))^nt^{n\beta}}{\Gamma(n\beta+1)}\\
&=c_1E_\beta(\kappa\Gamma(\beta) t^\beta).
\end{aligned}
\end{equation*}
Keeping in mind that we are interested in the behaviour as $\kappa$ tends to infinity while $t$ is fixed, we can apply Proposition \ref{boundMLf} to obtain the result.
\end{proof}

We have the ``converse" of the above result.

\begin{proposition}\label{prop:lower-renew}
Let $T\leq \infty$ and $\beta>0$.  Suppose that $f(t)$ is a nonnegative locally integrable function satisfying

\begin{equation}\label{renew:lower}
f(t)\geq c_2+\kappa\int_0^t(t-s)^{\beta-1}f(s)\d s\quad\text{for all}\quad0\leq t\leq T,
\end{equation}
where $c_2$ is some positive number.  Then for any $t\in(0,T]$, we have the following

\begin{equation*}
\liminf_{\kappa\rightarrow \infty}\frac{\log \log f(t)}{\log \kappa}\geq \frac{1}{\beta}.
\end{equation*}

\end{proposition}

\begin{proof}
With the notations introduced in the proof of Proposition \ref{prop:upper-renew}, \eqref{renew:lower} yields
\begin{equation}
f(t)\geq c_2\sum_{k=0}^{n-1}(\sA^k1)(t)+(\sA^n f)(t).
\end{equation}
Now similar arguments as in Proposition \ref{prop:upper-renew} prove the result. We leave it to the reader to fill in the details.
\end{proof}

 The above inequalities are well studied; see for instance \cite{Henry}.  But the novelty here is that, as opposed to what is usually done, instead of $t$, we take $\kappa$ to be large.

\section{Proofs of Theorem \ref{white} and Corollary \ref{cor:white}}
We will begin with the proof of Theorem \ref{white}.  We will prove it in two steps.  Set
\begin{equation}\label{sup:sol}
\sS_t(\lambda):=\sup_{x\in B(0,\,R)}\E|u_t(x)|^2.
\end{equation}
We then have the following proposition.

\begin{proposition}
Fix $t>0$, then
\begin{equation*}
\limsup_{\lambda\rightarrow \infty}\frac{\log \log \sS_t(\lambda)}{\log \lambda}\leq \frac{2\alpha}{\alpha-1}.
\end{equation*}
\end{proposition}

\begin{proof}
We start off with the representation \eqref{mild:dirichlet} and take the second moment to obtain
\begin{equation}\label{mild:white}
\begin{aligned}
\E|u_t(x)|^2&=|(\sG_Du)_t(x)|^2+\lambda^2\int_0^t\int_{B(0,\,R)}p^2_D(t-s,\,x,\,y)\E|\sigma(u_s(y))|^2\d y\d s\\
&=I_1+I_2.
\end{aligned}
\end{equation}
Clearly, for any fixed $t>0$, $I_1\leq c_1$ where $c_1$ is a constant depending on $t$.   We now focus our attention on $I_2$. The Lipschitz property of $\sigma$ together with the Markov property of killed stable processes yield the following

\begin{equation*}
\begin{aligned}
I_2&\leq (\lambda L_\sigma)^2\int_0^t\int_{B(0,\,R)}p^2_D(t-s,\,x,\,y)\E|u_s(y)|^2\d y\d s\\
&\leq (\lambda L_\sigma)^2\int_0^t\sS_s(\lambda)\int_{B(0,\,R)}p^2_D(t-s,\,x,\,y)\d y\d s\\
&\leq (\lambda L_\sigma)^2\int_0^t\sS_s(\lambda) p_D(2(t-s),\,x,\,x)\d s\\
&\leq c_2\lambda^2\int_0^t\frac{\sS_s(\lambda)}{(t-s)^{1/\alpha}}\d s.
\end{aligned}
\end{equation*}
Putting these estimates together, we have

\begin{equation*}
\sS_t(\lambda)\leq c_1+c_2\lambda^2\int_0^t\frac{\sS_s(\lambda)}{(t-s)^{1/\alpha}}\d s.
\end{equation*}
Now application of Proposition \ref{prop:upper-renew} proves the result.
\end{proof}

For any fixed $\ve>0$, set

\begin{equation*}
\sI_{\ve, t}(\lambda):=\inf_{x\in B(0,\,R-\ve)}\E|u_t(x)|^2.
\end{equation*}

\begin{proposition}
For any fixed $\ve>0$, there exists a $t_0>0$ such that for all $t\leq t_0$,

\begin{equation*}
\liminf_{\lambda\rightarrow \infty}\frac{\log \log \sI_{\ve, t}(\lambda)}{\log \lambda}\geq \frac{2\alpha}{\alpha-1}.
\end{equation*}
\end{proposition}
\begin{proof}
As in the proof of the previous proposition we start off with \eqref{mild:white} and seek to find lower bound on each of the terms. We fix $\ve>0$ and choose $t_0$ as in Proposition \ref{lower}. For $x\in B(0,\,R-\ve)$, we have $\sG_D(t,\,x)\geq g_{t_0}$. Hence $I_1\geq g_{t_0}^2$.  We now turn our attention to $I_2$.

\begin{equation*}
\begin{aligned}
I_2&\geq (\lambda l_\sigma)^2\int_0^t\int_{B(0,\,R)}p^2_D(t-s,\,x,\,y)\E|u_s(y)|^2\d y\d s\\
&\geq (\lambda l_\sigma)^2\int_0^t\sI_{\ve, s}(\lambda)\int_{B(0,\,R-\ve)}p^2_D(t-s,\,x,\,y)\d y\d s\\
\end{aligned}\\
\end{equation*}
Set $A:=\{y\in B(0,\,R-\ve);|x-y|\leq (t-s)^{1/\alpha}\}$. Since $t-s\leq t_0$, we have $|A|\geq c_1(t-s)^{1/\alpha}$. Now using Proposition \ref{lower}, we have
\begin{equation*}
\begin{aligned}
\int_{B(0,\,R-\ve)}p^2_D(t-s,\,x,\,y)\d y& \geq c_2\int_{A}\frac{1}{(t-s)^{2/\alpha}}\d y\\
&=c_3\frac{1}{(t-s)^{1/\alpha}}.
\end{aligned}
\end{equation*}
We thus have
\begin{equation*}
I_2\geq c_4\lambda^2\int_0^t\frac{\sI_{\ve, s}(\lambda)}{(t-s)^{1/\alpha}}\d s.
\end{equation*}
Combining the above estimates, we have

\begin{equation*}
\sI_{\ve, t}\geq g_{t_0}^2+c_4\lambda^2\int_0^t\frac{\sI_{\ve, s}(\lambda)}{(t-s)^{1/\alpha}}\d s.
\end{equation*}
We now apply Proposition \ref{prop:lower-renew} to obtain the result.
\end{proof}
{\it Proof of Theorem \ref{white}.}
The proof of the result when $t\leq t_0$ follows easily from the above two propositions. To prove the theorem for all $t>0$, we only need to prove the above proposition for all $t>0$.  For any fixed $T,\,t>0$, by changing the variable we have
\begin{equation*}
\begin{aligned}
\E|&u_{T+t}(x)|^2\\
&=|(\sG_Du)_{T+t}(x)|^2+\lambda^2\int_0^{T+t}\int_{B(0,\,R)}p^2_D(T+t-s,\,x,\,y)\E|\sigma(u_s(y))|^2\d y\d s\\
&=|(\sG_Du)_{T+t}(x)|^2+\lambda^2\int_0^{T}\int_{B(0,\,R)}p^2_D(T+t-s,\,x,\,y)\E|\sigma(u_s(y))|^2\d y\d s\\
&+\lambda^2\int_0^{t}\int_{B(0,\,R)}p^2_D(t-s,\,x,\,y)\E|\sigma(u_{T+s}(y))|^2\d y\d s.\\
\end{aligned}
\end{equation*}
This gives us
\begin{equation*}
\E|u_{T+t}(x)|^2\geq|(\sG_Du)_{T+t}(x)|^2+ \lambda^2l_\sigma^2\int_0^{t}\int_{B(0,\,R)}p^2_D(t-s,\,x,\,y)\E|u_{T+s}(y)|^2\d y\d s.
\end{equation*}
Since $|(\sG_Du)_{T+t}(x)|^2$ is strictly positive, we can use the proof of the above proposition with an obvious modification to conclude that

\begin{equation*}
\limsup_{\lambda\rightarrow \infty}\frac{\log\log \E|u_{T+t}(x)|^2}{\log \lambda}\geq \frac{2\alpha}{\alpha-1},
\end{equation*}
for $x\in B(0,\,R-\ve)$ and small $t$.

\qed

{\it Proof of Corollary \ref{cor:white}.}  Note that
\begin{equation*}
\int_{-R}^R\E|u_t(x)|^2\d x\leq 2R\sup_{x\in [-R,\,R]}\E|u_t(x)|^2
\end{equation*}
and
\begin{equation*}
\int_{-R}^R\E|u_t(x)|^2\d x\geq 2(R-\ve)\inf_{x\in [-(R-\ve),\,R-\ve]}\E|u_t(x)|^2.
\end{equation*}

We now apply Theorem \ref{white} and use the definition of $\sE_t(\lambda)$ to obtain the result. \qed

\section{Proofs of Theorem \ref{coloured} and Corollary \ref{cor:coloured}}
Recall that
\begin{equation*}
\sS_t(\lambda):=\sup_{x\in B(0,\,R)}\E|u_t(x)|^2,
\end{equation*}
where here and throughout the rest of this section, $u_t$ will denote the solution to the \eqref{eq:dirichlet:colored}.
The following lemma will be crucial later. In what follows $f$ denotes the spatial correlation of the noise $\dot{F}$.

\begin{lemma}\label{integral-upper}
For all $x,y\in B(0,\,R)$,
\begin{equation}\label{int-upper}
\iint_{B(0,\,R)\times B(0,\,R)}p_D(t,\,x,\,w)p_D(t,\,y,\,z)f(w,z)\d w \d z\leq \frac{c_1}{t^{\beta/\alpha}},
\end{equation}
for some positive constant $c_1$.
\end{lemma}
\begin{proof}
We begin by noting that
\begin{equation*}
\begin{aligned}
\iint_{B(0,\,R)\times B(0,\,R)}&p_D(t,\,x,\,w)p_D(t,\,y,\,z)f(w-z,0)\d w \d z\\
&\leq\iint p(t,\,x,\,w)p(t,\,y,\,z)f(w-z,0)\d w \d z\\
&\leq \int p(2t,\,w,\,x-y) f(w,0) \d w.
\end{aligned}
\end{equation*}
Now the scaling property of the heat kernel and a proper change of variable proves the result.
\end{proof}

\begin{proposition}
Fix $t>0$, then
\begin{equation*}
\limsup_{\lambda\rightarrow \infty}\frac{\log \log \sS_t(\lambda)}{\log \lambda}\leq \frac{2\alpha}{\alpha-\beta}.
\end{equation*}
\end{proposition}

\begin{proof}
We start with the mild formulation to the solution to \eqref{eq:dirichlet:colored} which after taking the second moment gives us
\begin{equation*}
\begin{aligned}
\E|&u(t,\,x)|^2=|(\sG_Du)_t(x)|^2\\
&+\lambda^2\int_0^t\int_{B(0,\,R)\times B(0,\,R)}p_D(t-s,x,y)p_D(t-s,x,z)f(y,z)\E[\sigma(u_s(y))\sigma(u_s(z))]\d y\d z \d s\\
&=I_1+I_2.
\end{aligned}
\end{equation*}
We obviously have $I_1\leq c_1$.  Note that the Lipschitz assumption on $\sigma$ together with H\"older's inequality give
\begin{equation*}
\begin{aligned}
\E[\sigma(u_s(y))\sigma(u_s(z))]&\leq L_\sigma^2 [\E|u_s(y)|^2]^{1/2}[\E|u_s(z)|^2]^{1/2}\\
&\leq L_\sigma^2S_s(\lambda).
\end{aligned}
\end{equation*}
We can use the above inequality and Lemma \ref{integral-upper} to bound $I_2$ as follows.
\begin{equation*}
\begin{aligned}
I_2\leq (\lambda L_\sigma)^2\int_0^t\frac{\sS_s(\lambda)}{(t-s)^{\beta/\alpha}}\d s.
\end{aligned}
\end{equation*}
Combining the above estimates, we obtain
\begin{equation*}
\sS_t(\lambda)\leq c_1+c_2\lambda^2\int_0^t\frac{\sS_s(\lambda)}{(t-s)^{\beta/\alpha}}\d s,
\end{equation*}
which immediately yields the result upon an application of Proposition \ref{prop:upper-renew}.
\end{proof}

We have the following lower bound on the second of the solution. Inspired by the localisation arguments of \cite{Khoshnevisan:2013ab}, we have the following.
\begin{proposition}\label{prop:recur}
Fix $\ve>0$. Then for all $x\in B(0,\,R-2\ve)$ and $t\leq t_0$,
\begin{equation*}
\begin{aligned}
\E|u_t&(x)|^2\\
&\geq g_t^2+g_t^2\sum_{k=1}^\infty (\lambda l_\sigma c_1)^{2k}\left(\frac{t}{k} \right)^{k(\alpha-\beta)/\alpha},
\end{aligned}
\end{equation*}
where $c_1$ is some positive constant depending on $\alpha$ and $\beta$.
\end{proposition}
\begin{proof}
Fix $\ve>0$ and for convenience, set $B:=B(0,\,R)$ and $B_\ve:=B(0,\,R-\ve)$.  We will also use the following notation; $B^2:=B\times B$ and $B^2_\ve:=B_\ve\times B_\ve$.  After taking the second moment, the mild formulation of the solution together with the growth condition on $\sigma$ gives us
\begin{equation*}
\begin{aligned}
\E|u_t&(x)|^2\\
&\geq|(\sG_Du)_t(x)|^2+\lambda^2 l_\sigma^2\int_0^{t}\int_{B^2}\\
& p_D(t-s_1,\,x,\,z_1)p_D(t-s_1,\,x,\,z_1')\E|u_{s_1}(z_1)u_{s_1}(z_1')|f(z_1,z_1')\d z_1\d z_1'\d s_1.
\end{aligned}
\end{equation*}
We also have
\begin{equation*}
\begin{aligned}
&\E|u_{s_1}(z_1)u_{s_1}(z_1')|\geq|(\sG_Du)_{s_1}(z_1)(\sG_Du)_{s_1}(z_1)|+\lambda^2l_\sigma^2\int_0^{s_1}\int_{B^2} \\
&p_D(s_1-s_2,\,z_1,\,z_2)p_D(s_1-s_2,\,z_1',\,z_2')\E|u_{s_2}(z_2)u_{s_2}(z_2')|f(z_2, z_2')\d z_2\d z_2'\d s_2.
\end{aligned}
\end{equation*}
The above two inequalities thus give us
\begin{equation}
\begin{aligned}
\E|u_t&(x)|^2\\
&\geq|(\sG_Du)_t(x)|^2+\lambda^2 l_\sigma^2\int_0^{t}\int_{B^2}\\
& p_D(t-s_1,\,x,\,z_1)p_D(t-s_1,\,x,\,z_1')\E|u_{s_1}(z_1)u_{s_1}(z_1')|f(z_1,z_1')\d z_1\d z_1'\d s_1\\
&\geq |(\sG_Du)_t(x)|^2+\lambda^2l_\sigma^2\int_0^{t}\int_{B^2}\\
&p_D(t-s_1,\,x,\,z_1)p_D(t-s_1,\,x,\,z_1')f(z_1,z_1')(\sG_Du)_{s_1}(z_1)(\sG_Du)_{s_1}(z_1')\d z_1\d z_1'\d s_1\\
&+(\lambda l_\sigma)^4\int_0^{t}\int_{B^2}p(t-s_1,\,x,\,z_1)p(t-s_1,\,x,\,z_1')f(z_1,z_1')\int_0^{s_1}\int_{B^2}\\
&p_D(s_1-s_2,\,z_1,\,z_2)p_D(s_1-s_2,\,z_1,'\,z_2')
\E|u_{s_2}(z_2)u_{s_2}(z_2')|f(z_2, z_2')\d z_2\d z_2'\d s_2\d z_1\d z_1'\d s_1.
\end{aligned}
\end{equation}

We set $z_0=z_0':=x$ and $s_0:=t$ and continue the recursion as above to obtain
\begin{equation}\label{recursion}
\begin{aligned}
\E|u_t&(x)|^2\\
&\geq|(\sG_Du)_t(x)|^2\\
&+ \sum_{k=1}^\infty (\lambda l_\sigma)^{2k}\int_0^t\int_{B^2}\int_0^{s_1}\int_{B^2}\cdots \int_0^{s_{k-1}}\int_{B^2} |(\sG_Du)_{s_k}(z_k)(\sG_Du_{s_k})(z_k')|\\
&\prod_{i=1}^kp_D(s_{i-1}-s_{i}, z_{i-1},\,z_i)p_D(s_{i-1}-s_{i}, z'_{i-1},\,z'_i)f(z_i, z_i') \d z_i\d z_i'\d s_i.
\end{aligned}
\end{equation}
Therefore,
\begin{equation*}
\begin{aligned}
\E|u_t&(x)|^2\\
&\geq|(\sG_Du)_t(x)|^2\\
&+ \sum_{k=1}^\infty (\lambda l_\sigma)^{2k}\int_0^t\int_{B_\ve^2}\int_0^{s_1}\int_{B_\ve^2}\cdots \int_0^{s_{k-1}}\int_{B_\ve^2} |(\sG_Du)_{s_k}(z_k)(\sG_Du)_{s_k}(z_k')|\\
&\prod_{i=1}^kp_D(s_{i-1}-s_{i}, z_{i-1},\,z_i)p_D(s_{i-1}-s_{i}, z'_{i-1},\,z'_i)f(z_i, z_i') \d z_i\d z_i'\d s_i.
\end{aligned}
\end{equation*}
Using the fact that for $z_k,\,z_k'\in B_\ve$,
\begin{equation*}
\begin{aligned}
(\sG_Du)_{s_k}(z_k)&(\sG_Du)_{s_k}(z_k')\\
&\geq \inf_{x,y\in B_\ve}\inf_{0\leq s\leq t}(\sG_Du)_s(x)(\sG_Du)_s(y)\\
&=g_t^2,
\end{aligned}
\end{equation*}
we obtain

\begin{equation*}
\begin{aligned}
\E|u_t&(x)|^2\\
&\geq g_t^2+g_t^2\sum_{k=1}^\infty (\lambda l_\sigma)^{2k}\int_0^t\int_{B_\ve^2}\int_0^{s_1}\int_{B_\ve^2}\cdots \int_0^{s_{k-1}}\int_{B_\ve^2} \\
&\prod_{i=1}^k p_D(s_{i-1}-s_{i}, z_{i-1},\,z_i)p_D(s_{i-1}-s_{i}, z'_{i-1},\,z'_i)f(z_i, z_i') \d z_i\d z_i'\d s_i.
\end{aligned}
\end{equation*}
We reduce the temporal region of integration as follows.
\begin{equation*}
\begin{aligned}
\E|u_t&(x)|^2\\
&\geq g_t^2+g_t^2\sum_{k=1}^\infty (\lambda l_\sigma)^{2k}\int_{t-t/k}^t\int_{B_\ve^2}\int_{s_1-t/k}^{s_1}\int_{B_\ve^2}\cdots \int_{s_{k-1}-t/k}^{s_{k-1}}\int_{B_\ve^2} \\
&\prod_{i=1}^k p_D(s_{i-1}-s_{i}, z_{i-1},\,z_i)p_D(s_{i-1}-s_{i}, z'_{i-1},\,z'_i)f(z_i, z_i') \d z_i\d z_i'\d s_i.
\end{aligned}
\end{equation*}

Now we make a change the temporal variable, $s_{i-1}-s_{i} \rightarrow s_{i}$, in the following way such that for all integers $i \in [1,k] $, we have 
\begin{equation*}
\begin{aligned}
&\int_{s_{i-1}-t/k}^{s_{i-1}} p_D(s_{i-1}-s_{i}, z_{i-1},\,z_i)p_D(s_{i-1}-s_{i}, z'_{i-1},\,z'_i)f(z_i, z_i') \d s_i \\
&= \int_0^{t/k} p_D(s_{i}, z_{i-1},\,z_i)p_D(s_{i}, z'_{i-1},\,z'_i)f(z_i, z_i') \d s_i.
\end{aligned}
\end{equation*}
We thus have 
\begin{equation*}
\begin{aligned}
\E|u_t&(x)|^2\\
&\geq g_t^2+g_t^2\sum_{k=1}^\infty (\lambda l_\sigma)^{2k}\int_0^{t/k}\int_{B_\ve^2}\int_0^{t/k}\int_{B_\ve^2}\cdots \int_0^{t/k}\int_{B_\ve^2} \\
&\prod_{i=1}^kp_D(s_i, z_{i-1},\,z_i)p_D(s_i, z'_{i-1},\,z'_i)f(z_i, z_i') \d z_i\d z_i'\d s_i.
\end{aligned}
\end{equation*}

We now focus our attention on the multiple integral appearing in the above inequality. We will further restrict its spatial domain of integration so that we have the required lower bound on each component of the following product,

\begin{equation}\label{product}
\prod_{i=1}^kp_D(s_i, z_{i-1},\,z_i)p_D(s_i, z'_{i-1},\,z'_i)f(z_i, z_i').
\end{equation}

Recall that $x\in B(0,\,R-2\ve)$.  For each $i=1,\cdots, k$, choose $z_i$ and $z_i'$ satisfying

\begin{equation*}
z_i\in B(z_0, s_1^{1/\alpha}/2)\cap B(z_{i-1}, s_i^{1/\alpha})
\end{equation*}
and
\begin{equation*}
z_i'\in B(z'_0, s_1^{1/\alpha}/2)\cap B(z'_{i-1}, s_i^{1/\alpha}),
\end{equation*}
so that we have $|z_i-z_i'|\leq s_i^{1/\alpha}/2$ together with $|z_i-z_{i-1}|\leq s_i^{1/\alpha}$ and $|z'_i-z'_{i-1}|\leq s_i^{1/\alpha}$.  Now using Proposition \ref{lower}, we can conclude that $p_D(s_i, z_{i-1},\,z_i)\geq s_i^{-d/\alpha}$ and $p_D(s_i, z_{i-1}',\,z_i')\geq s_i^{-d/\alpha}$.  Moreover, we have $|z_i-z_i'|\leq s_1^{1/\alpha}$, which gives us $f(z_i, z_i')\geq s_1^{-\beta/\alpha}$.  In other words, we are looking at the points $\{s_i,\,z_i,\,z_i' \}_{i=0}^k$ such that the following holds
\begin{equation*}
\prod_{i=1}^kp_D(s_i, z_{i-1},\,z_i)p_D(s_i, z'_{i-1},\,z'_i)f(z_i, z_i')\geq \prod_{i=1}^k\frac{1}{s_i^{2d/\alpha}s_1^{\beta/\alpha}}.
\end{equation*}
Note that we have $|B(x,\,s_1^{1/\alpha}/2)\cap B(z_{i-1}, s_i^{1/\alpha})|\geq c|s_i|^{d/\alpha}$  and $|B(x,\,s_1^{1/\alpha}/2)\cap B(z'_{i-1}, s_i^{1/\alpha})|\geq c|s_i|^{d/\alpha}$ for some constant $c$, independent of $i$.   For notational convenience we set $\sA_i:=\{z_i\in B(x,\,s_1^{1/\alpha}/2)\cap B(z_{i-1}, s_i^{1/\alpha})\}$ and $\sA_i':=\{z_i'\in B(x,\,s_1^{1/\alpha}/2)\cap B(z'_{i-1}, s_i^{1/\alpha})\}$ which lead us to
\begin{equation*}
\begin{aligned}
\int_0^{t/k}&\int_{\sA_1}\int_{\sA'_1}\int_0^{t/k}\int_{\sA_2}\int_{\sA'_2}\cdots \int_0^{t/k}\int_{\sA_k}\int_{\sA'_k} \\
&\prod_{i=1}^kp_D(s_i, z_{i-1},\,z_i)p_D(s_i, z'_{i-1},\,z'_i)f(z_i, z_i') \d z_i\d z_i'\d s_i\\
&\geq \int_0^{t/k}\int_{\sA_1}\int_{\sA'_1}\int_0^{t/k}\int_{\sA_2}\int_{\sA'_2}\cdots \int_0^{t/k}\int_{\sA_k}\int_{\sA'_k} \\
&\prod_{i=1}^k\frac{1}{s_i^{2d/\alpha}s_1^{\beta/\alpha}}\d z_i\d z_i'\d s_i.
\end{aligned}
\end{equation*}

We now use the lower bounds on the area of $\sA'_i$s and $\sA_i$s to estimate the spatial integrals and then evaluate the time integrals to end up with the following lower bound on the above quantity
\begin{equation*}
\begin{aligned}
 c^{2k}\int_0^{t/k}& s_1^{k-1}\frac{1}{s_1^{k\beta/\alpha}}\d s_1\\
&=\frac{c^{2k}}{2^{\alpha(k-1)}}\frac{\alpha}{k(\alpha-\beta)}\left(\frac{t}{k} \right)^{k(\alpha-\beta)/\alpha}.
\end{aligned}
\end{equation*}
Putting the above estimates together we obtain

\begin{equation*}
\begin{aligned}
\E|u_t(&x)|^2\\
&\geq g_t^2+g_t^2\sum_{k=1}^\infty (\lambda l_\sigma)^{2k}\frac{c^{2k}}{2^{\alpha(k-1)}}\frac{\alpha}{k(\alpha-\beta)}\left(\frac{t}{k} \right)^{k(\alpha-\beta)/\alpha}\\
&\geq g_t^2+g_t^2\sum_{k=1}^\infty (\lambda l_\sigma c_1)^{2k}\left(\frac{t}{k} \right)^{k(\alpha-\beta)/\alpha},
\end{aligned}
\end{equation*}
for some constant $c_1$.  
\end{proof}

Recall that

\begin{equation*}
\sI_{\ve, t}(\lambda):=\inf_{x\in B(0,\,R-\ve)}\E|u_t(x)|^2,
\end{equation*}
where here  $u_t$ is the solution to \eqref{eq:dirichlet:colored}. We now have

\begin{proposition}
For any fixed $\ve>0$, there exists a $t_0>0$ such that for all $t\leq t_0$,
\begin{equation*}
\liminf_{\lambda\rightarrow \infty}\frac{\log \log \sI_{\ve, t}(\lambda)}{\log \lambda}\geq \frac{2\alpha}{\alpha-\beta}.
\end{equation*}
\end{proposition}
\begin{proof}
We begin by writing

\begin{equation*}
\begin{aligned}
\sum_{k=1}^\infty (\lambda l_\sigma c_1)^{2k}&\left(\frac{t}{k} \right)^{k(\alpha-\beta)/\alpha}\\
&=\sum_{k=1}^\infty \left(\frac{(\lambda l_\sigma c_1)^2t^{(\alpha-\beta)/\alpha}}{k^{(\alpha-\beta)/\alpha}} \right)^{k}.
\end{aligned}
\end{equation*}
Lemma \ref{sum} with $\rho:=(\alpha-\beta)/\alpha$ and $\theta:=\lambda^2$ together with the above proposition finishes the proof.
\end{proof}
{\it Proof of Theorem \ref{coloured}:}
The above two propositions prove the theorem for all $t\leq t_0$. We now extend the result to all $t>0$. As in the proof of Theorem \ref{white}, we only need to extend the above proposition to any fixed $t>0$. For any $T,\,t>0$,
\begin{equation*}
\begin{aligned}
\E|u_{T+t}&(x)|^2\\
&\geq|(\sG_Du)_{t+T}(x)|^2+\lambda^2 l_\sigma^2\int_0^{T+t}\int_{B^2}\\
& p_D(T+t-s_1,\,x,\,z_1)p_D(T+t-s_1,\,x,\,z_1')\E|u_{s_1}(z_1)u_{s_1}(z_1')|f(z_1,z_1')\d z_1\d z_1'\d s_1.
\end{aligned}
\end{equation*}
 This leads to
\begin{equation*}
\begin{aligned}
\E|u_{T+t}&(x)|^2\\
&\geq|(\sG_Du)_{t+T}(x)|^2+\lambda^2 l_\sigma^2\int_0^{t}\int_{B^2}\\
& p_D(t-s_1,\,x,\,z_1)p_D(t-s_1,\,x,\,z_1')\E|u_{T+s_1}(z_1)u_{T+s_1}(z_1')|f(z_1,z_1')\d z_1\d z_1'\d s_1.
\end{aligned}
\end{equation*}
A similar argument to that used in the proof of Proposition \ref{prop:recur} shows that
\begin{equation*}
\begin{aligned}
\E|u_{T+t}&(x)|^2\\
&\geq|(\sG_Du)_{T+t}(x)|^2\\
&+ \sum_{k=1}^\infty (\lambda l_\sigma)^{2k}\int_0^t\int_{B^2}\int_0^{s_1}\int_{B^2}\cdots \int_0^{s_{k-1}}\int_{B^2} |(\sG_Du)_{T+s_k}(z_k)(\sG_Du)_{T+s_k}(z_k')|\\
&\prod_{i=1}^kp_D(s_{i-1}-s_{i}, z_{i-1},\,z_i)p_D(s_{i-1}-s_{i}, z'_{i-1},\,z'_i)f(z_i, z_i') \d z_i\d z_i'\d s_i.
\end{aligned}
\end{equation*}
Similar ideas to those used in the rest of the proof of Proposition \ref{prop:recur} together with the proof of the above proposition show that for all $t\leq t_0$, we have
\begin{equation*}
\liminf_{\lambda\rightarrow \infty}\frac{\log \log \E|u_{T+t}(x)|^2}{\log \lambda}\geq \frac{2\alpha}{\alpha-\beta}.
\end{equation*}
for all $T>0$ and whenever $x\in B(0,\,R-\ve)$.

\qed

{\it Proof of Corollary \ref{cor:coloured}:} The proof is exactly the same as that of Corollary \ref{cor:white} and is omitted.\qed

\section{Some extensions.}

We begin this section by showing that the methods developed in this can be used to study the stochastic wave equation as well. More precisely, we give an alternative proof of a very interesting result proved in \cite{Khoshnevisan:2013ab}. Consider the following equation 
\begin{equation}\label{wave}
\partial_{tt} u_t(x)=\partial_{xx} u_t(x)+\lambda \sigma(u_t(x))\dot{w}(t,\,x) \quad\text{for}\quad x\in \R \quad t>0,
\end{equation}
with initial condition $u_0(x)=0$ and non-random initial velocity $v_0$ satisfying $v_0\in L^1(\R)\cap L^2(\R)$ and $\|v_0\|_{L^2(\R)}>0$. As before $\sigma$ satisfies the conditions mentioned in the introduction. We set $\sE_t(\lambda):=\sqrt{\int_{-\infty}^\infty\E|u_t(x)|^2\,\d x}$ and restate the result of \cite{Khoshnevisan:2013ab} as follows,
\begin{theorem}
Fix $t>0$, we then have
\begin{equation*}
\lim_{\lambda\rightarrow \infty}\frac{\log \log \sE_t(\lambda)}{\log \lambda}=1
\end{equation*}

\end{theorem}

\begin{proof}
We again use the theory of Walsh \cite{walsh} to make sense of \eqref{wave} as the solution to the following integral equation

\begin{equation*}
u_t(x)=\frac{1}{2}\int_{-t}^tv_0(x-y)\,\d y+\frac{1}{2}\lambda \int_0^t\int_{\R}1_{[0,t-s]}(|x-y|)\sigma(u_s(y))W(\d s\d y).
\end{equation*} 

We now use Walsh's isometry to obtain 

\begin{equation*}
\begin{aligned}
\E|u_t(x)|^2&=\frac{1}{4}\left| \int_{-t}^tv_0(x-y)\,\d y\right|^2\\
&+\frac{1}{4}\lambda^2\int_0^t\int_{\R}1_{[0,t-s]}(|x-y|)\E |\sigma(u_s(y)|^2\,\d s\,\d y.
\end{aligned}
\end{equation*}

Recall that from the assumption on the initial velocity, we have 
\begin{equation*}
\int_\R\left| \int_{-t}^tv_0(x-y)\,\d y\right|^2\,\d x\leq 4t^2\|v_0\|^2_{L^2(\R)}.
\end{equation*}

This and the assumption on $\sigma$ yields 

\begin{equation}\label{upper-wave}
\sE^2_t(\lambda)\leq 4t^2\|v_0\|_{L^2(\R)}+\frac{1}{4}\lambda^2L^2_\sigma\int_0^t(t-s)\sE^2_s(\lambda)\,\d s.
\end{equation}

Using similar ideas we can obtain the following lower bound,
\begin{equation}\label{lower-wave}
\sE^2_t(\lambda)\geq t^2\|v_0\|_{L^2(\R)}+\frac{1}{4}\lambda^2L^2_\sigma\int_0^t(t-s)\sE^2_s(\lambda)\,\d s.
\end{equation}
We now use Propositions \ref{prop:upper-renew} and \ref{prop:lower-renew} together with the above two inequalities to obtain the result.
\end{proof}

The method developed so far can be adapted to the study of a much wider class of stochastic heat equations, once we have the ``right" heat kernel estimates. Indeed, \eqref{heat:upper} and \eqref{heat:lower} were two crucial elements of our method. So by considering operators whose heat kernels behave in a nice way, we can generate examples of stochastic heat equations for which, we can apply our method. Recall that we are considering equations of the type,
\begin{equation}\label{general-dirichlet}
\partial_t u_t(x)=\sL u_t(x)+\lambda \sigma(u_t(x))\dot{F}(t,\,x).\\
\end{equation}

In what follows, we will choose different $\sL$s while keeping all the other conditions as before. And again, the choice of these operators $\sL$s will make the boundary conditions clear.  Some of the equations below appear to be new.  We again to not prove existence-uniqueness results as these are fairly standard once we have a grip on the heat kernel.  See \cite{walsh} and \cite{minicourse}.

\begin{example}
We choose $\sL$ to be the generator of a Brownian motion defined on the interval $(0,1)$ which is reflected at the point $1$ and killed at the other end of the interval. So, we are in fact looking at 
\begin{equation*}
\left|\begin{split}
&\partial_t u_t(x)=\frac{1}{2}\partial _{xx}u_t(x)+\lambda \sigma(u_t(x))\dot{F}(t,\,x)\quad\text{for}\quad0<x<1\quad\text{and}\quad t>0\\
&u_t(0)=0,\quad \partial_{x}u_t(1)=1 \quad \text{for}\quad t>0.
\end{split}
\right.
\end{equation*}
It can be shown that for any $\ve>0$, there exists a $t_0>0$, such that for all $x\in [\epsilon, 1)$ and $t\leq t_0$, the heat kernel of this Brownian motion satisfies 
\begin{equation*}
p(t,\,x,\,y)\asymp t^{-d/2},
\end{equation*}
whenever $|x-y|\leq t^{1/2}$. We use the method developed in this paper to conclude that
\begin{equation*}
\lim_{\lambda\rightarrow \infty}\frac{\log \log \E|u_t(x)|^2}{\log \lambda}=\frac{4}{2-\beta},
\end{equation*}
whenever $x\in [\ve,1)$.
\end{example}

\begin{example}
Let $X_t$ be censored stable process as introduced in \cite{BoBurChen}.  These have been studied in \cite{Chen-Kim-Song}.  Roughly speaking, the censored stable process in the ball $B(0,\,R)$ can be obtained by suppressing the jump from $B(0,\,R)$ to the complement of $B(0,\,R)^c$.  The process is thus forced to stay inside $B(0,\,R)$. We denote the generator of this process by $-(-\Delta)^{\alpha/2}|_{B(0,\,R)}$ and consider the following equation 
\begin{equation}\label{regional}
\partial_t u_t(x)=-(-\Delta)^{\alpha/2}|_{B(0,\,R)} u_t(x)+\lambda \sigma(u_t(x))\dot{F}(t,\,x),\\
\end{equation}
In a sense, the above the above equation can be regarded as fractional equation with Neumann boundary condition.  In \cite{Chen-Kim-Song}, it was shown that the probability density function of $X_t$, which we denote by $\bar{p}(t,\,x,\,y)$ satisfies
\begin{equation*}
\bar{p}(t,\,x,\,y)\asymp\left(1\wedge \frac{\delta^{\alpha/2}_{B(0,\,R)}(x)}{t^{1/2}}\right)\left(1\wedge \frac{\delta^{\alpha/2}_{B(0,\,R)}(y)}{t^{1/2}}\right)p(t,\,x,\,y),
\end{equation*}
So we can proceed as in the proof of Theorem \ref{coloured} to see that we have 
\begin{equation}\label{conclusion}
\lim_{\lambda\rightarrow \infty}\frac{\log \log \E|u_t(x)|^2}{\log \lambda}=\frac{2\alpha}{\alpha-\beta}.
\end{equation}
\end{example}

\begin{example}
In this example, we choose $\sL$ be the generator of the relativistic stable process killed upon exiting the ball $B(0,\,R)$. We are therefore looking at the following equation 
\begin{equation*}
\left|\begin{split}
&\partial_t u_t(x)= mu_t(x)-(m^{2/\alpha}-\Delta)^{\alpha/2}u_t(x)+\lambda \sigma(u_t(x))\dot{F}(t,\,x),\\
&u_t(x)=0, \quad \text{for all}\quad x\in B(0,\,R)^c.
\end{split}
\right.
\end{equation*}
Here $m$ is some fixed positive number. One can show that for any $\ve>0$, there exists a $t_0>0$, such that for all $x,y\in B(0,\,R-\ve)$ and $t\leq t_0$, we have
\begin{equation*}
p(t,\,x,\,y)\asymp t^{-d/\alpha},
\end{equation*}
whenever $|x-y|\leq t^{1/\alpha}$. See for instance \cite{Chen-Kim-Song2}.  The constants involved in the above inequality depends on $m$.
We therefore have the same conclusion as that of Theorem \ref{coloured}. In other words, we have 
\begin{equation}\label{conclusion}
\lim_{\lambda\rightarrow \infty}\frac{\log \log \E|u_t(x)|^2}{\log \lambda}=\frac{2\alpha}{\alpha-\beta},
\end{equation}
whenever $x\in B(0,\R-\ve)$.
\end{example}

\begin{example}
Let $1<\bar{\alpha}<\alpha<2$ and consider the following 
\begin{equation*}
\left|\begin{split}
&\partial_t u_t(x)= -(-\Delta)^{\alpha/2}u_t(x)-(-\Delta)^{\bar{\alpha}/2}u_t(x)+\lambda \sigma(u_t(x))\dot{F}(t,\,x),\\
&u_t(x)=0, \quad \text{for all}\quad x\in B(0,\,R)^c.
\end{split}
\right.
\end{equation*}
The Dirichlet heat kernel for the operator $\sL:=-(-\Delta)^{\alpha/2}-(-\Delta)^{\bar{\alpha}/2}$ has been studied in \cite{Chen-Kim-Song3}. Since $\bar{\alpha}\leq \alpha$, it is known that for small times, the behaviour of the heat kernel estimates is dominated by the fractional Laplacian $-(-\Delta)^{\alpha/2}$. More precisely, for any $\ve>0$, there exists a $t_0>0$, such that for all $x,y\in B(0,\,R-\ve)$ and $t\leq t_0$, we have
\begin{equation*}
p(t,\,x,\,y)\asymp t^{-d/\alpha},
\end{equation*}
whenever $|x-y|\leq t^{1/\alpha}$.  Therefore, in this case also, we have \eqref{conclusion}.
\end{example}



\bibliography{FoonKhosh}

\def\cprime{$'$} \def\polhk#1{\setbox0=\hbox{#1}{\ooalign{\hidewidth
  \lower1.5ex\hbox{`}\hidewidth\crcr\unhbox0}}}
  \def\polhk#1{\setbox0=\hbox{#1}{\ooalign{\hidewidth
  \lower1.5ex\hbox{`}\hidewidth\crcr\unhbox0}}} \def\cprime{$'$}
\begin{thebibliography}{10}

\bibitem{BoBurChen}
Krzysztof Bogdan, Krzysztof Burdzy, and Zhen-Qing Chen.
\newblock Censored stable processes.
\newblock {\em Probab. Theory Related Fields}, 127(1), 2003.

\bibitem{Bogdan}
Krzysztof Bogdan, Tomasz Grzywny, and Micha{\l} Ryznar.
\newblock Heat kernel estimates for the fractional {L}aplacian with {D}irichlet
  conditions.
\newblock {\em Ann. Probab.}, 38(5):1901--1923, 2010.

\bibitem{Chen-Kim-Song3}
Zhen-Qing Chen, Panki Kim, and Renming Song.
\newblock Dirichlet heat kernel estimates for
  {$\Delta^{\alpha/2}+\Delta^{\beta/2}$}.
\newblock {\em Illinois J. Math.}, 54(4), 2010.

\bibitem{Chen-Kim-Song}
Zhen-Qing Chen, Panki Kim, and Renming Song.
\newblock Two-sided heat kernel estimates for censored stable-like processes.
\newblock {\em Probab. Theory Related Fields}, 146(3-4), 2010.

\bibitem{Chen-Kim-Song2}
Zhen-Qing Chen, Panki Kim, and Renming Song.
\newblock Sharp heat kernel estimates for relativistic stable processes in open
  sets.
\newblock {\em Ann. Probab.}, 40(1), 2012.

\bibitem{minicourse}
Robert Dalang, Davar Khoshnevisan, Carl Mueller, David Nualart, and Yimin Xiao.
\newblock {\em A minicourse on stochastic partial differential equations},
  volume 1962 of {\em Lecture Notes in Mathematics}.
\newblock Springer-Verlag, Berlin, 2009.

\bibitem{ferrante}
Marco Ferrante and Marta Sanz-Sol{{\'e}}.
\newblock S{PDE}s with coloured noise: analytic and stochastic approaches.
\newblock {\em ESAIM Probab. Stat.}, 10:380--405 (electronic), 2006.

\bibitem{foonJose}
Mohammud Foondun and Mathew Joseph.
\newblock Remarks on non-linear noise excitability of some stochastic
  equations.
\newblock {\em Stochastic Process. Appl.}, to appear.

\bibitem{Henry}
Daniel Henry.
\newblock {\em Geometric theory of semilinear parabolic equations}, volume 840
  of {\em Lecture Notes in Mathematics}.
\newblock Springer-Verlag, Berlin, 1981.

\bibitem{Khoshnevisan:2013aa}
Davar Khoshnevisan and Kunwoo Kim.
\newblock Non-linear noise excitation of intermittent stochastic pdes and the
  topology of lca groups.
\newblock {\em Ann. Probab.}, to appear, 02.

\bibitem{Khoshnevisan:2013ab}
Davar Khoshnevisan and Kunwoo Kim.
\newblock Non-linear noise excitation and intermittency under high disorder.
\newblock {\em Proc. AMS}, 02 2013.

\bibitem{Mainardi}
F.~Mainardi.
\newblock Fractional calculus: some basic problems in continuum and statistical
  mechanics.
\newblock In {\em Fractals and fractional calculus in continuum mechanics
  ({U}dine, 1996)}, volume 378 of {\em CISM Courses and Lectures}, pages
  291--348. Springer, Vienna, 1997.

\bibitem{walsh}
John~B. Walsh.
\newblock An {I}ntroduction to {S}tochastic {P}artial {D}ifferential
  {E}quations.
\newblock In {\em \'{E}cole d'\'et\'e de {P}robabilit\'es de {S}aint-{F}lour,
  {XIV}---1984}, volume 1180 of {\em Lecture Notes in Math.}, pages 265--439.
  Springer, Berlin, 1986.

\end{thebibliography}
\end{document}